\documentclass[12pt,a4paper]{amsart}

\usepackage{amsfonts, amsmath, amssymb, amsthm, hyperref}

\newtheorem{thm}{Theorem}
\newtheorem*{thm*}{Theorem}
\newtheorem{lem}{Lemma}
\newtheorem{con}{Conjecture}
\newtheorem{cor}{Corollaries}

\theoremstyle{definition}
\newtheorem{defn}{Definition}

\theoremstyle{remark}

\DeclareMathOperator{\conv}{conv}
\DeclareMathOperator{\cent}{cent}
\DeclareMathOperator{\suppo}{supp}

\DeclareMathOperator{\bd}{bd} \DeclareMathOperator{\cl}{cl}

\newcommand{\pt}{{\rm pt}}
\renewcommand{\int}{\mathop{\rm int}}

\renewcommand{\epsilon}{\varepsilon}

\begin{document}
\title{Dual central point theorems and their generalizations}
\author{R.N.~Karasev}

\begin{abstract}
We prove some analogues of the central point theorem and Tverberg's theorem, where instead of points, we consider hyperplanes or affine flats of given dimension.
\end{abstract}

\subjclass[2000]{52A35,52C35,55M35}
\keywords{Tverberg's theorem, central point theorem}

\maketitle

\section{Introduction}

Let us state some classical theorems of convex and combinatorial geometry, that are generalized and extended in this paper.

The first statement is the Neumann-Rado theorem~\cite{neumann1945,rado1946,grun1960} (see also the reviews~\cite{eck1993} and~\cite{dgk1963}) about the central point of a measure.

\begin{thm*}[The central point theorem]
For an absolutely continuous probabilistic measure $\mu$ on $\mathbb R^d$ there exists a point $x\in \mathbb R^d$ such that for any halfspace $H\ni x$ 
$$
\mu(H)\ge\dfrac{1}{d+1}.
$$
\end{thm*}

There is also a discrete version of this theorem. 

\begin{thm*}[The discrete central point theorem]
Suppose $X\subset\mathbb R^d$ is a finite set with $|X|=n$. Then there exists $x\in\mathbb R^d$ such that for any halfspace $H\ni x$
$$
|H\cap X| \ge \left\lfloor\dfrac{n+d}{d+1}\right\rfloor.
$$
\end{thm*}

An important generalization of this theorem is proved in~\cite{tver1966}.

\begin{thm*}[Tverberg's theorem]
Consider a set $X\in \mathbb R^d$ with $|X|=(d+1)(n-1)+1$. Then $X$ can be partitioned into $n$ subsets $X_1, \ldots, X_n$ so that
$$
\bigcap_{i=1}^n \conv X_i\not=\emptyset.
$$
\end{thm*}

In the papers~\cite{zivvre1990,dol1993} the central point theorem was extended to the case of several measures.

\begin{thm*}[The central transversal theorem]
Suppose $m+1$ absolutely continuous probabilistic measures $\mu_0,\ldots, \mu_m$ are given on $\mathbb R^d$. Then there exist an affine $m$-flat $L\in\mathbb R^d$ such that for any halfspace $H\supseteq L$ and any $i=0,\ldots,m$ 
$$
\mu_i(H)\ge\dfrac{1}{d-m+1}.
$$
\end{thm*}

In the paper~\cite{zivvre1994} a topological approach was applied to prove the following analogue of Tverberg's theorem.

\begin{thm*}[The colorful Tverberg theorem]
Let a subset $X\subset\mathbb R^d$ consist of $|X|=(d+1)t$ points, where $t\ge 2r-1$, $r=p^k$, $p$ is a prime. Let the points of $X$ be colored into $d+1$ colors, each color having $t$ points exactly.

Then we can choose $r$ disjoint subsets $X_1,\ldots, X_r\subset X$ so that the following conditions hold. For any $i=1,\ldots, r$ we have $|X_i|=d+1$, $X_i$ has all the $d+1$ colors. And
$$
\bigcap_{i=1}^r \conv X_i \not=\emptyset.
$$
\end{thm*}

The proofs of the central transversal theorem and the colorful Tverberg theorem made essential use of topology and calculating some obstructions. In this paper similar methods are used.

The paper is organized as follows.

\begin{itemize}

\item
Section~\ref{results} contains the statements of the main results. 

\item
In Sections~\ref{measures}, \ref{centralpoints} we remind some definitions and lemmas from the measure theory, and prove some theorems on the set of central points of a given measure. 

\item
In Section~\ref{fixedpoints} we generalize the Brouwer fixed point theorem for families of fiberwise maps of a vector bundle, Sections~\ref{releuler}, \ref{fixedproofs} contain the respective proofs.

\item
In Sections~\ref{fpcplemma}, \ref{cptproofs} we give proofs for the dual central point and the dual central transversal theorems. 

\item
Sections~\ref{tvgconstr}, \ref{tvtconstr}, \ref{tvproofs} contain the proofs of Tverberg type theorems. 

\item
Section~\ref{conjectures} contains some conjectures, related to the results of the paper.
\end{itemize}

\section{Main results}
\label{results}

In this paper we prove some analogues of the central point theorem, where we consider hyperplanes or affine flats of given dimension instead of points. Note that some similar results for hyperplanes can be found in~\cite{ruhu1999}.

It should be noted that these results are not obtained from the original theorems for points by applying a polar transform or some other kind of duality. Let us give a short explanation: if we make a polar transform with center $0$ and then apply the corresponding theorem for points, then the inverse polar transform should by applied in another point (the central point) to keep the statement true. But the composition of such two transforms is not affine, it is projective in general and does not retain the relevant convexity structures. 

The proofs are mostly based on calculating some topological obstructions to nonzero sections of vector bundles, these obstructions being the Euler classes.

Let us state the dual central point theorem in the discrete case.

\begin{thm}
\label{dualcpt} Suppose $\mathcal F$ is a family of $n$ hyperplanes of general position in $\mathbb R^d$. Then there exists a point $x$ such that any ray starting at $x$ intersects at least 
$$
\left\lfloor\dfrac{n+d}{d+1}\right\rfloor
$$ 
hyperplanes of $\mathcal F$.
\end{thm}

The general position of a family of hyperplanes in $\mathbb R^d$ means that any $d$ of them have exactly one common point, and any $d+1$ of them have no point in common.

There is also a central point theorem for measures on the set of hyperplanes. We need some definitions first.

\begin{defn}
We need the following identification of the set of affine $k$-flats in $\mathbb R^d$ with the total space $\gamma_d^{d-k}\to G_d^{d-k}$ of the canonical bundle over the Grassmann manifold of $d-k$-subspaces in $\mathbb R^d$. 

Any affine $k$-flat $L$ corresponds to the unique $d-k$-dimensional linear subspace $V$, orthogonal to $L$, and the unique point $V\cap L\in V$.
\end{defn}

\begin{defn}
Denote the set of affine $k$-flats, that intersect a given subset $X\subseteq\mathbb R^d$ by $I(X, k)$.
\end{defn}

Now we are ready to state the dual central point theorem for measures.

\begin{thm}
\label{dualcptmes} Consider an absolute continuous probabilistic measure $\mu$ on $\gamma_d^1$ with compact support. Then there exists a point $x$ such that for any ray $r$ starting at $x$ we have
$$
\mu (I(r, d-1))\ge\frac{1}{d+1}.
$$
\end{thm}

\begin{defn}
A \emph{$k$-half-flat} in $\mathbb R^d$ is a nonempty proper subset of a $k$-flat $L$, given by some linear inequality $l(x)\ge 0$. The $k-1$-flat $\{x\in L : l(x) = 0\}$ is called \emph{the boundary} of this half-flat.
\end{defn}

Let us state the dual central transversal theorem.

\begin{thm}
\label{dualctr} Suppose $d-k$ absolutely continuous probabilistic measures $\mu_1,\ldots,\mu_{d-k}$ are given on $\gamma_d^{d-k}$, all the measures having compact supports. Then there exists a $d-k-1$-flat $L$ such that for any $d-k$-half-flat $M$, bounded by $L$, and any $i=1,\ldots,d-k$
$$
\mu_i (I(M, k))\ge \frac{1}{k+2}.
$$
\end{thm}

Theorem~\ref{dualcpt} implies the dual Tverberg theorem in the plane, such results were known before, see~\cite{roud1988} for example.

\begin{cor}
\label{dualtverberg2} Consider a family of $3n$ straight lines in general position in the plane. Then they can be partitioned into $n$ triples so that all the triangles, corresponding to the triples, have a common point.
\end{cor}

We are going to study the dual Tverberg theorem in arbitrary dimension. Let us make a definition.

\begin{defn}
We say that some $d+1$ hyperplanes $h_1,\ldots,h_{d+1}$ of general position in $\mathbb R^d$ \emph{form a simplex} $S$, if $S$ is the convex hull of the finite point set $\{x_i\}$, defined as
$$
x_i = \bigcap_{j\neq i} h_j.
$$
It is obvious that the facets of $S$ are subsets of the respective hyperplanes $h_i$.
\end{defn}

We conjecture the following statement.

\begin{con}[The dual Tverberg theorem]
Suppose $\mathcal F$ is a family of $(d+1)n$ hyperplanes in general position in $\mathbb R^d$. Then $\mathcal F$ can be partitioned into $n$ subfamilies of $d+1$ hyperplanes each so that the simplexes, formed by the subfamilies, have a common point.
\end{con}

In this paper we are going to prove this conjecture for prime powers $n$. This is an essential constraint of the topological method, as it is in the proof of the topological Tverberg theorem in~\cite{vol1996}, for example.

\begin{thm}
\label{dualtverbergp} Suppose $\mathcal F$ is a family of $(d+1)n$ hyperplanes in general position in $\mathbb R^d$, $n$ being a prime power.

Then $\mathcal F$ can be partitioned into $n$ subfamilies of $d+1$ hyperplanes each, so that the simplexes, formed by the subfamilies, have a common interior point.
\end{thm}

There is also a dual version of the colorful Tverberg theorem from~\cite{zivvre1994}.

\begin{thm}
\label{coldualtverberg} Consider a family $\mathcal F$ of $(d+1)t$ hyperplanes in general position in $\mathbb R^d$. Let $t\ge 2r-1$, $r=p^k$, $p$ be a prime. Suppose $\mathcal F$ is colored into $d+1$ colors, each color having exactly $t$ elements.

Then we can choose $r$ disjoint subfamilies $\mathcal F_1,\ldots,\mathcal F_r\subset \mathcal F$ of $d+1$ hyperplanes each so that all the simplexes, formed by the subfamilies, have a common interior point, and every subfamily $\mathcal F_i$ is colored into all the $d+1$ colors.
\end{thm}

\section{Some properties of measures}
\label{measures}

Be the Radon--Nikodym theorem (see~\cite{shg1977}) absolutely continuous measures on $\mathbb R^d$ are given by nonnegative functions of class $L_1$. We also consider measures on smooth manifolds, given by $L_1$ functions in any coordinate patch. 

In the sequel we consider absolutely continuous measures on a manifold $X$, mostly we consider \emph{probabilistic} measures, i.e. the measures $\mu$ such that $\mu(X)=1$.

In this Section we give some definitions and prove technical lemmas. Denote $\bd X$ and $\cl X$ the boundary and the closure of a set $X\subseteq\mathbb R^d$.

\begin{defn}
Let $P$ be a topological space. The family of measures $\mu_p$ on a manifold $X$ \emph{depends continuously} on $p\in P$, if for any open set $U\subseteq X$ the number $\mu_p(U)$ depends continuously on $p$.
\end{defn}

\begin{defn}
\emph{The support} of a measure $\mu$ on $X$ is the set
$$
\suppo \mu = \{x\in X : \text{for any neighborhood}\ U\ni x \
\mu(U)
> 0\}.
$$
The support of a measure is obviously closed.
\end{defn}

\begin{defn}
For any locally trivial bundle of manifolds $f : X\to Y$ a measure $\mu$ with compact support on $X$ defines a \emph{projection} of measure $f_*\mu$ on $Y$ by the formula
$$
f_*\mu (A) = \mu (f^{-1}(A)).
$$
\end{defn}

\begin{lem}
The measure $f_*\mu$ is locally given by an $L_1$-function. If $\mu_p$ depends continuously on $p$ then $f_*\mu_p$ depends continuously on $p$ too.
\end{lem}

\begin{proof}
The first claim follows from Fubini's theorem (see~\cite{shg1977}), the second is obvious by definition.
\end{proof}

We also need a continuous dependence of a measure projection on the projection map.

\begin{defn}
Let us define the \emph{compact-open} topology on the set of maps $f :X\to Y$. Put the base neighborhoods to be
$$
U_{K,V} = \{f :X\to Y : f(K)\subseteq V\}
$$
for any compact $K\subseteq X$ and open $V\subseteq Y$.
\end{defn}

\begin{lem}
\label{mesprojcont0}
Let a family of maps $f_p : X \to Y$ between manifolds depend continuously on the parameter $p\in P$ in the compact-open topology. Let every $f_p$ be a locally trivial bundle.

Then for any $\mu$ with compact support on $X$ the measures ${f_p}_*\mu$ depend continuously on $p\in P$.
\end{lem}

\begin{proof}
Let $C=\suppo \mu$. For any $q\in P$ take a precompact neighborhood $V\supseteq f_q(C)$. Then by the definition of compact-open continuity, there exists a neighborhood $U\ni q$ such that $f_p(C)\subseteq V$ for all $p\in U$. Hence, we can assume the supports of ${f_p}_*(\mu)$ to be contained in a single compact set.

Take an open set $W\subseteq Y$. Let us show that $\mu(f_p^{-1}(W))$ depends continuously on $p$ in $q$. As it was mentioned, the set $W$ may be considered precompact. Put $W' = f_q^{-1}(W)$ and find a compact set $C_1$ and an open set $C_2$ so that 
$$
C_1\subseteq W'\subseteq\cl W'\subseteq C_2\cap C\subseteq C
$$ 
and 
$$
|\mu(C_1) - \mu(W')| , |\mu(C_2) - \mu(W')| < \varepsilon.
$$
This can be done since $\mu$ is given by $L_1$-function.

Then $f_p(C_1)\subseteq W$ and $f_p(C\setminus C_2)\cap\cl W=\emptyset$ for any $p$ in some neighborhood of $q$ by the definition of comact-open topology. Hence for such $p$ the sets $f_p^{-1}(W)$ are between $C_1$ and $C_2$ (by inclusion), and their measures differ from $\mu(W')$ by at most $\varepsilon$.
\end{proof}

We also need a generalization of the previous lemma to the case, when the maps are not locally trivial bundles, but they are locally trivial bundles almost everywhere.

\begin{lem}
\label{mesprojcont} Consider a family of maps $f_p : X\setminus S_p \to Y$, where $S_p$ are some closed subsets, depending on $p\in P$. Let every $f_p$ be a locally trivial bundle on its domain. Consider also a measure $\mu$ on $X$ with compact support. Suppose for any $q\in P$
$$
\forall\varepsilon>0\ \exists\ \text{a neighborhood}\ U\ni q :
\mu(\bigcup_{p\in U} S_p) < \epsilon.
$$
Let the map $f_p$ restricted to $X\setminus \bigcup_{p\in U} S_p$ depend continuously on $p\in U$.

Then the measures ${f_p}_*\mu$ depend continuously on $p\in P$.
\end{lem}

\begin{proof}
For any $q\in P$ take its respective neighborhood $U$ and consider $X'=X\setminus \bigcup_{p\in U} S_p$. Now apply Lemma~\ref{mesprojcont0}, the measure projections are changed by at most $\varepsilon$ by going to $X'$, hence the continuity follows by going to the limit $\varepsilon\to 0$.
\end{proof}

\begin{defn}
Call a measure \emph{connected} if its support is connected.
\end{defn}

\begin{defn}
Call a measure $\mu$ with compact support on $\mathbb R^d$ \emph{$1$-convex} if the support of every its orthogonal projection to a curve is connected.
\end{defn}

\begin{defn}
\emph{A layer} in $\mathbb R^d$ is the set of points between two parallel hyperplanes, including the hyperplanes. The layer \emph{width} is the distance between these hyperplanes.
\end{defn}

We need some lemmas about the measure of layers and halfspaces.

\begin{lem}
\label{layermes} Consider a family of probabilistic measures $\{\mu_p\}_{p\in P}$ on $\mathbb R^d$, parameterized by a bicompact set $P$. Let all the supports of measures be contained in a single compact set. Then for any $\varepsilon>0$ there exists $\delta > 0$ such that for any layer $L$ of width $<\delta$ and any $p\in P$
$$
\mu_p(L) < \varepsilon.
$$
\end{lem}

\begin{proof}
Let all the supports be contained in some compact $S$.

Assume the contrary. Then there is a sequence of layers $L_i$ and values $p_i\in P$ such that the width of $L_i$ tends to zero, while the measures $\mu_{p_i}(L_i) > \varepsilon$. Since all the layers touch $S$, then the set of all possible layers is compact, we may assume that they tend to some hyperplane $H$. From the bicompactness we may also assume that $p_i$ tend to $p$. Now take small enough $\delta$ so that $\mu_p(H+B_\delta)$ ($\delta$-neighborhood of $H$) is less that $\varepsilon/2$. The sets $L_i\cap S$ are contained in $H+B_\delta$ for large enough $i$, therefore the measures $\mu_{p_i}(H+B_\delta)$ have upper limit $\varepsilon/2$, that is a contradiction.
\end{proof}

\begin{lem}
\label{halfspacemes}
Suppose a family of probabilistic measures $\{\mu_p\}_{p\in P}$ on $\mathbb R^d$ depends continuously on $p\in P$, let the topology of $P$ have countable base. Let the supports of all $\mu_p$ be contained in a single compact set $S$. Then the measure of a halfspace $\mu_p(H)$ depends continuously on the pair $(p, H)$.
\end{lem}

A halfspace can be defined by the inequality $(n, x) \ge c$, where $n\in S^{d-1}, c\in\mathbb R$. So we parameterize halfspaces by $S^{d-1}\times\mathbb R$, and consider the continuity w.r.t. a halfspace under this identification.

\begin{proof}
Consider a sequence of halfspaces $H_i\to H$ and values $p_i\to p$. Let $H$ be defined by the inequality 
$$
(n, x) \ge c,
$$
put 
$$
H_\delta^- = \{x\in\mathbb R^d : (n,x) \ge c+\delta\},\quad
H_\delta^+ = \{x\in\mathbb R^d : (n,x) \ge c-\delta\}.
$$
For small enough $\delta$ the measures $\mu_p(H_\delta^+)$, $\mu_p(H_\delta^-)$ and $\mu_p(H)$ differ at most by $\varepsilon$. Moreover, for large enough $i$ we have
$$
S\cap H_\delta^- \subseteq S\cap H_i\subseteq S\cap H_\delta^+.
$$

Also, for large enough $i$ the differences $\mu_{p_i}(H) - \mu_p(H)$, $\mu_{p_i}(H_\delta^-) - \mu_p(H_\delta^-)$ and $\mu_{p_i}(H_\delta^+) - \mu_p(H_\delta^+)$ have absolute value at most $\varepsilon$. 

If follows now that $\mu_{p_i}(H_i)$ differ from $\mu_p(H)$ by at most $2\varepsilon$ for large enough $i$, that is exactly what is needed.
\end{proof}

\begin{lem}
\label{minmescont}
Let a family of probabilistic measures $\{\mu_p\}_{p\in P}$ on $\mathbb R^d$ depend continuously on $p\in P$, let the topology of $P$ have countable base. Let the supports of all $\mu_p$ be contained in a single compact set. Then the following function (the minimum is taken over halfspaces)
$$
f_p(y) = \min_{H\ni y} \mu_p (H)
$$
depends continuously on $p$ and $y$. 
\end{lem}

\begin{proof}
When taking the minimum, it is sufficient to consider the case, when $H$ contains $y$ on its boundary. In this case all possible $H$ are parameterized by the $d-1$-dimensional sphere, that is a compact set. 

Now it remains to note that by Lemma~\ref{halfspacemes} the expression to be minimized is a continuous function on $P\times\mathbb R^d\times S^{d-1}$, the minimum is taken over a compact $S^{d-1}$, and therefore it is continuous.
\end{proof}

\section{Central points of measures}
\label{centralpoints}

\begin{defn}
A point $x$ is a \emph{central point} of a measure $\mu$ on $\mathbb R^d$ if for any halfspace $H\ni x$
$$
\mu(H) \ge \frac{\mu(\mathbb R^d)}{d+1}.
$$
Denote the set of all central points for $\mu$ by $\cent\mu$.
\end{defn}

We need a lemma on the structure of the central point set for a measure with compact support. Of course, it is nonempty by the central point theorem of~\cite{neumann1945,rado1946}. In the sequel all the measures are probabilistic, absolutely continuous, and if they depend on a parameter $p\in P$, then the parameter set has a countable base of its topology.

\begin{lem}
\label{cptsetstruct} Let $\mu$ be a probabilistic measure with compact support on
$\mathbb R^d$. Then either $\cent\mu$ is a convex body (with nonempty interior), or a single point.
\end{lem}

\begin{proof}
Note that
$$
\cent\mu = \bigcap\{H : H\ \text{is a halfspace and}\ \mu(H)\ge \frac{d}{d+1}\}.
$$
This proves that $\cent\mu$ is a convex compact set. Let us call a halfspace $H$ ``bad'' if $\mu(H)\le \dfrac{1}{d+1}$.

Any point $y\not\in \cent\mu$ is contained in some bad halfspace. Let us tend a point $y$ to some $x\in\bd(\cent\mu)$. From the compactness of the space of normals to a hyperplane and Lemma~\ref{halfspacemes} we obtain, that the limit point $x$ is also in some bad halfspace.

Assume that the interior of $\cent\mu$ is empty and $\cent\mu$ is not a single point. Denote the affine hull of $\cent\mu$ by $L$. The space $\mathbb R^d$ is covered by bad halfspaces. Moreover, it is covered by bad halfspaces, whose boundary contain $L$. Denote the latter set of halfspaces by $\mathcal H$.

Consider the quotient space $\mathbb R^d/L$ and the projection of $\mu$ to it. Take some $\varepsilon < \dfrac{1}{d(d+1)}$ and $\delta$ from Lemma~\ref{layermes}. Now consider the projections to $\mathbb R^d/L$ of the following sets: the complements of $\delta$-neighborhoods of all halfspaces in $\mathcal H$, let their projections form the family $\mathcal G$. 

The intersection $\bigcap \mathcal G$ is empty, and therefore some $d$ of those halfspaces have an empty intersection by Helly's theorem. Consider now the preimages of those $d$ halfspaces in $\mathbb R^d$, they cover $\mathbb R^d$ up to $d$ layers of width
$\delta$. But the latter is impossible since
$$
\mu(\mathbb R^d) < \frac{d}{d+1} + d\varepsilon < 1,
$$
that is a contradiction.
\end{proof}

We need some claims about the continuity of $\cent\mu$ as a function of $\mu$. Denote an open ball of radius $\varepsilon$ centered at the origin by $D_\varepsilon$, denote its closure by $B_\varepsilon$. We also denote the $\varepsilon$-neighborhood of a set $X\subseteq\mathbb R^d$ by
$$
X+D_\varepsilon = \{x+y: x\in X,\ y\in D_\varepsilon\}.
$$

\begin{lem}
\label{cptcont1} Suppose a family of $1$-convex probabilistic measures $\mu_p$ on $\mathbb R^d$ depends on $p\in P$ continuously, and all their supports are contained in some compact set. Then for any $q\in P$
$$
\forall \varepsilon>0\ \exists\ \text{a neighborhood}\ U\ni q : p\in U\implies
\cent\mu_p\subseteq \cent\mu_q + D_\varepsilon.
$$
\end{lem}

\begin{proof}
Let $\cent\mu_q = C$.

Take some $\varepsilon > 0$. For any point $y\not\in C+D_{\varepsilon}$ there exists a halfspace $H$ such that $H\cap C=\emptyset$, $y\in\int H$, and $\mu_q (H) = \dfrac{1}{d+1}$. It follows from $1$-convexity that we can assume $\mu_q (H) <
\dfrac{1}{d+1}$. 

Consider a continuous function $f_q(y)$ of Lemma~\ref{minmescont}. Note that $f_q(y) < \dfrac{1}{d+1}$ on $\mathbb R^d\setminus C + D_\varepsilon$. Note that $f_q(y)=0$ outside the convex hull of $\suppo \mu_q$, and $f_q(y)$ depends continuously on $y$. Then its maximum on $\mathbb R^d\setminus C + D_\varepsilon$ is attained somewhere and is at most $\dfrac{1}{d+1}$. Hence the inequality $f_p(y) < \dfrac{1}{d+1}$ keeps for this $y\in \mathbb R^d\setminus C + D_\varepsilon$ and some neighborhood of $q$ in $P$. That is what we need.
\end{proof}

\begin{lem}
\label{cptcont2} Suppose a family of $1$-convex probabilistic measures $\mu_p$ on $\mathbb R^d$ depends on $p\in P$ continuously, and all their supports are contained in some compact set. Then for any $q\in P$
$$
\forall \varepsilon>0\ \exists\ \text{a neighborhood}\ U\ni q : p\in U\implies \cent\mu_p + D_\varepsilon\supseteq \cent\mu_q.
$$
\end{lem}

\begin{proof}
If $\cent\mu_q$ is a single point, then the statement follows from Lemma~\ref{cptcont1}. Otherwise $C=\cent\mu_q$ has a nonempty interior by Lemma~\ref{cptsetstruct}.

Consider a closed convex set $C'\subset \int C$ such that $C'+D_\varepsilon \supseteq C$. Consider the function $f_q(y)$ from Lemma~\ref{minmescont}, that is a continuous function of $q$ and $y$. The $1$-convexity implies that on $C'$ the function $f_q$ is strictly greater than $\dfrac{1}{d+1}$. Hence, for some neighborhood $U\ni q$ for any $p\in U$ this inequality still holds, and $C'\subseteq \cent\mu_p$ for $p\in U$.
\end{proof}

Actually the two previous lemmas show that the set $\cent\mu$ depends continuously on $\mu$ in the Hausdorff metric, provided that the supports of all the measures are contained in a single compact set, and all the measures are $1$-convex.

\section{Fixed point theorems for fiberwise maps of vector bundles}
\label{fixedpoints}

We have to generalize the Brouwer fixed point theorem for the case, when there are several fiberwise maps of a vector bundle and we search for a common fixed point of them.

We assume that every vector bundle $\eta$ has a continuous inner product on its fibers. In this case we denote the spaces of unit spheres and balls respectively by $S(\eta)$, $B(\eta)$. The coefficients of the cohomology are $Z_2$ in this section.

\begin{thm}
\label{fpbundles} Suppose we have a vector bundle $\eta\to X$, the fiber dimension $\dim\eta = n$, and a number $k\ge 1$. Also suppose that the topmost Stiefel-Whitney class $w_n(\eta)\in H^*(X)$ has the property $w_n(\eta)^k\neq 0$. 

Now consider $k+1$ fiberwise continuous maps $f_i : B(\eta) \to B(\eta)$ $(i=1,\ldots, k+1)$. Then there exists a point $x\in B(\eta)$ such that
$$
x = f_1(x) = \ldots = f_{k+1}(x).
$$
\end{thm}

In practice, the following version of this theorem may be more useful.

\begin{thm}
\label{fpbundles2} Suppose we have a vector bundle $\eta\to X$, the fiber dimension $\dim\eta = n$, and a number $k\ge 1$. Also suppose that the topmost Stiefel-Whitney class $w_n(\eta)\in H^*(X)$ has the property $w_n(\eta)^k\neq 0$. 

Now consider $k+1$ fiberwise continuous maps $f_i : \eta\to\eta$ $(i=1,\ldots, k+1)$ and a number $M\in\mathbb R$ with the following property: for vectors $x\in \eta$ with $|x|\ge M$ we have
$$
(x, f_1(x) - x) < 0.
$$
Then there exists a point $x\in \eta$ such that
$$
x = f_1(x) = \ldots = f_{k+1}(x).
$$
\end{thm}

Note that in this theorem only one of the maps satisfies the condition, that guarantees a fixed point in every fiber. The other maps are arbitrary continuous fiberwise maps.

Note that if the space $X$ consists of one point, then Theorems~\ref{fpbundles} and \ref{fpbundles2} give the classical Brouwer theorem.

In the sequel we apply Theorem~\ref{fpbundles2} to the canonical vector bundle $\gamma_d^k\to G_d^k$ over the Grassmannian, and we need a lemma that was used, for example in~\cite{dol1993} in similar situations, and was known before, see~\cite{hil1980A} for example.

\begin{thm}
\label{fpgrass} Consider the canonical bundle over the Grassmanian $\gamma_d^k\to G_d^k$. Then we have 
$$
w_k(\gamma_d^k)^{d-k}\not=0\in H^*(G_d^k, Z_2).
$$
\end{thm}

\section{The relative Euler class}
\label{releuler}

We need a notion of the Euler class of a vector bundle in the relative cohomology of a pair to prove Theorem~\ref{fpbundles}. Some general information on vector bundles and their topology can be found in the textbooks~\cite{mishch1998, milsta1974}. The notion of the relative Euler class is frequently used implicitly in topological reasonings, here we fix some its properties.

Unless otherwise stated, for oriented bundles we consider the integer coefficients of cohomology, and for non-oriented bundles we consider coefficients $Z_2$.

\begin{defn}
Consider a pair $Y\subseteq X$ and an $m$-dimensional vector bundle $\eta\to X$. Suppose we have a section $s$ of $\eta|_Y$, with no zeros. We call such a construction $(\eta, s)$ a \emph{partial section}.
\end{defn}

We can also assume that the section $s$ is extended somehow to the entire $X$, allowing zeros on $X\setminus Y$. It is obvious from the convexity of fibers that all such extensions are homotopic, the homotopy having no zeros on $Y\times[0, 1]$.

The partial sections over $(X, Y)$ are classified by the maps of the pair $(X, Y)$ to the pair $(BO(m), BO(m-1))$, or $(BSO(m), BSO(m-1))$ for oriented bundles. Generally, we consider them up to homotopy.

The pair $(BO(m), BO(m-1))$ has the following geometric realization. Consider the canonical $m$-vector bundle $\gamma\to BO(m)$. Now the pair $(B(\gamma), S(\gamma))$ realizes $(BO(m), BO(m-1))$ up to homotopy. Remind the Thom isomorphism. There is an element $u\in H^m(BO(m), BO(m-1))$ such that the multiplication of $H^*(BO(m))$ by $u$ gives an isomorphism of $Z_2$-cohomology
$$
H^*(BO(m), BO(m-1)) = uH^*(BO(m)).
$$
The same is true for $(BSO(m), BSO(m-1))$ and integer coefficient cohomology.

\begin{defn}
The image of the Thom class $u\in H^m(BO(m), BO(m-1), Z_2)$ in $H^m(X, Y,
Z_2)$ is called \emph{the Euler class modulo $2$} of the partial section. In the oriented case the image of $u\in H^m(BSO(m), BSO(m-1), \mathbb Z)$ in $H^m(X, Y, \mathbb Z)$ is called \emph{the Euler class} of the partial section. Denote this class $e(\eta, s)$.
\end{defn}

Denote $\pt$ the one-point space. The following lemma gives the connection between the relative Euler class and the ordinary Euler class.

\begin{lem}
For a pair $(X\sqcup\pt, \pt)$ the relative Euler class is the same as the ordinary Euler class, or the topmost Stiefel-Whitney class in the non-oriented case.
\end{lem}

It is more or less obvious from the definition that the relative Euler class is the first obstruction (possibly, modulo $2$) to extend a partial nonzero section to a complete nonzero section.

We need to take $\times$-products of relative sections.

\begin{defn}
Suppose there is a partial section $(\eta_1, s_1)$ over $(X_1, Y_1)$, and a partial section $(\eta_2, s_2)$ over $(X_2, Y_2)$. The sections $s_i$ can be extended over the respective $X_i$. Then in the $\times$-product bundle $\eta_1\times\eta_2$ the section $s_1+s_2$ is a partial section over the pair $(X_1\times X_2, (X_1\times Y_2)\cup (Y_1\times X_2))$. Call this bundle and section a \emph{$\times$-product of partial sections}.
\end{defn}

\begin{lem}
\label{eulermult} There is a formula for $\times$-product of partial sections
$$
e(\eta_1\times\eta_2, s_1+s_2) = e(\eta_1, s_1)\times e(\eta_2, s_2).
$$
\end{lem}

This lemma follows easily from considering the classifying spaces and the multiplicativity of the Thom class. If we consider the direct sum of two bundles over the same space, then we obtain the following lemma.

\begin{lem}
If we take the direct sum of two partial section: $(\eta_1,s_1)$ over $(X,Y_1)$ and $(\eta_2, s_2)$ over $(X,Y_2)$, then 
$$
e(\eta_1\oplus\eta_2, s_1\oplus s_2) = e(\eta_1, s_1)e(\eta_2, s_2)\in H^*(X, Y_1\cup Y_2).
$$
\end{lem}

\section{Proofs of the fixed point theorems}
\label{fixedproofs}

\begin{proof}[Proof of Theorem~\ref{fpbundles}]
Note that it is sufficient to prove the theorem for a compact $X$, since if the class  $w_n(\eta)^k\neq 0$ in $H^*(X)$, then it is nonzero in some the cohomology of some compact subspace of $X$.

First let us multiply the maps $f_i$ in each fiber by a homothety with factor $(1-\varepsilon)$. It is sufficient to prove the theorem for new maps and every $\varepsilon>0$, then it follows for $\epsilon=0$ from compactness. So we assume that $f_i : B(\eta)\to B(\eta)\setminus S(\eta)$.

Consider the bundle $\eta'$ induced from $\eta$ by the natural projection
$B(\eta)\to X$. Define the map $s_1 : B(\eta)\to \eta'$ by the formula $s_1 : x\mapsto x-f_1(x)$, it is a partial section $(\eta', s_1)$ over $(B(\eta), S(\eta))$. Moreover, this partial section is homotopy equivalent to the standard partial section $(\eta', s_0)$ defined by $s_0 : x\mapsto x$. Hence the class $e(\eta', s_1)$ (as $e(\eta', s_0)$) equals the Thom class $u(\eta)\in H^*(B(\eta), S(\eta))$, it follows straight from the definition of the relative Euler class and the explicit realization of $(BO(m), BO(m-1))$, described above.

Now by Lemma~\ref{eulermult} and Thom's isomorphism
$$
e(\eta'\oplus\eta'^k, s_1\oplus 0) = u(\eta)w(\eta')^k\neq 0.
$$
Hence, the section $s_1\oplus 0$ cannot be extended to a nonzero section of $\eta'\oplus\eta'^k$ over $B(\eta)$.

Consider the sections $s_i$ $(i=2,\ldots, k+1)$ of $\eta'$ over $B(\eta)$, defined by
$$
s_i : x\mapsto f_i(x) - f_1(x).
$$
The section $s_1\oplus s_2\oplus \ldots \oplus s_{k+1}$ is homotopy equivalent over $S(\eta)$ to the section $s_1\oplus 0$ (the homotopy $s_1\oplus ts_2\oplus \ldots \oplus ts_{k+1}$ for $t\in[0, 1]$ does the job). Hence, it cannot be extended to nonzero section of  $\eta'^{k+1}$ over $B(\eta)$, and in some point $x\in B(\eta)$ we have
$$
x = f_1(x) = \ldots = f_{k+1}(x).
$$
\end{proof}

\begin{proof}[Proof of theorem~\ref{fpbundles2}]
In this theorem we take the ball bundle in $\eta$, where balls are of radius $M$, then we act similar to the previous proof.

In fact all the reasonings are valid, since we did not actually use that the image of 
$f_i$ is in $B(\eta)$, except one place. It was used to prove that the section $s_1 : x\mapsto x-f_1(x)$ is homotopy equivalent to the section $s_0 : x\mapsto x$. In this theorem it is done by the simple homotopy $(1-t)s_0 + t s_1$, which has no zeros on $S(\eta)\times[0, 1]$.
\end{proof}

\section{Some lemmas needed to prove Theorem~\ref{dualcpt}}
\label{fpcplemma}

The following lemma generalizes the discrete central point theorem and the Brouwer theorem about fixed points. Denote the index set
$$
[n] = \{1, 2,\ldots, n\}.
$$

\begin{lem}
\label{cptfp} Let $B\subset\mathbb R^d$ be a convex compact set. Suppose there are $n$ maps $f_i : B\to\mathbb R^d$ $(i\in[n])$. Denote
$$
l =\left\lfloor\dfrac{n+d}{d+1}\right\rfloor.
$$
Suppose that for any $x\in B$ at most $l-1$ points of $f_i(x)$ are outside $B$. Then there exists $x\in B$ such that any halfspace $H\ni x$ contains at least $l$ points of
$\{f_i(x)\}_{i\in [n]}$.
\end{lem}

In other words, $x$ is a central point of $\{f_i(x)\}_{i\in [n]}$.

\begin{proof}
Denote for any $I\subseteq[n]$ the following point set
$$
f_I(x) = \{f_i(x) : i\in I\}.
$$
We have to find such $x$ that
$$
x \in \bigcap_{I\subseteq[n], |I|=n-l+1} \conv f_I(x).
$$
The intersection on the right side is nonempty by Helly's theorem and is contained in $B$ by the hypothesis. If the intersection is a continuous function of $x$ in Hausdorff metric, then a continuous selection (a function) is possible
$$
x\mapsto g(x)\in \bigcap_{I\subseteq[n], |I|=n-l+1} \conv f_I(x).
$$
In this case the Brouwer fixed point theorem gives the required point $x$. 

In the case when there is no Hausdorff metric continuity, we have to act more carefully.

Consider the sets (Minkowski sums with a small ball)
$$
C_I(x) = \conv f_I(x) + B_\varepsilon,
$$
these sets depend continuously on $x$ in Hausdorff metric, their intersection
$$
C(x) = \bigcap_{I\subseteq[n], |I|=n-l+1} C_I(x)\cap B
$$
has a nonempty interior and therefor depends continuously on $x$ in the Hausdorff metric.

Hence as above we obtain a point $x\in B$ such that $x\in C(x)$. When $\varepsilon$ tends to zero, from the compactness considerations we obtain a point $x\in\bigcap_{I\subseteq[n], |I|=n-l+1} \conv
f_I(x)$.
\end{proof}

We also need the following lemma from~\cite{adw1984}.

\begin{lem}
\label{projbounded}
Suppose $\mathcal F=\{h_1, \ldots, h_n\}$ is a set of hyperplanes in $\mathbb R^d$, consider the orthogonal projections $\pi_1, \ldots, \pi_n$ onto the respective hyperplanes. Then there exists a convex body $P$, such that 
$$
\forall i=1,\ldots,n,\ \pi_i(P)\subseteq P.
$$ 
\end{lem}

\section{Proofs of the central point theorems}
\label{cptproofs}

\begin{proof}[Proof of Theorem~\ref{dualcpt}]
Let $\mathcal F = \{h_1, h_2, \ldots, h_n\}$. For any $i\in [n]$ define the map $f_i : \mathbb R^d\to\mathbb R^d$ as the orthogonal projection of $x$ onto $h_i$. By Lemma~\ref{projbounded} there exists a compact convex body $B$ that is stable under projections $f_i$.

Now apply Lemma~\ref{cptfp} to the body $B$ and maps $f_i$. The lemma gives a point $x$ such that any halfspace $H\ni x$ contains at least $l=\left\lfloor\dfrac{n+d}{d+1}\right\rfloor$ points of $f_i(x)$.
Since $f_i(x)$ are the orthogonal projections, then for any ray $r$ starting at $x$ there are at least $l$ hyperplanes of $\mathcal F$ that either intersect $r$ or parallel to $r$.

Now we have to show that the ray actually intersects at least $l$ hyperplanes. Consider the unit sphere $S$ with center $x$. Let us project the hyperplanes $\mathcal F$ to the sphere, obtaining the family $\mathcal G$. The family $\mathcal G$ consists of either hemispheres or full spheres (if $x$ is on the corresponding $h\in\mathcal F$). It is already proved that the closures of $\mathcal G$ cover $S$ with multiplicity at least $l$ everywhere, and it is left to show that $\mathcal G$ itself covers $S$ with multiplicity $l$. 

Assume the contrary: $e\in S$ is covered by $k<l$ sets of $\mathcal G$ and is on the boundary of at least $l-k$ sets in $\mathcal G$. Denote the latter sets by $\mathcal H\subset\mathcal G$. It follows from the general position that $|\mathcal H| \le d-1$ and in any neighborhood of $e$ there exists a point that is not contained in any $\cl X$ of some $X\in\mathcal H$. In this case some close enough point $e'$ is contained in the same sets of $\mathcal G$ as $e$ (because the sets are open) and is not on the boundary of any set in $\mathcal G$. Hence, this point is contained in at most $k$ of the sets in $\mathcal G$, that is a contradiction.
\end{proof}

\begin{proof}[Proof of Corollary~\ref{dualtverberg2}]
Let us apply Theorem~\ref{dualcpt} to the family of lines, that gives some point $x$. Let us order the lines circularly $\{l_1, l_2, l_3,\ldots, l_{3n}\}$, corresponding to the order of projections of $x$ to these lines. If $x$ is on some line then we can take any direction as the direction towards this line.

Now the triples $(l_k, l_{k+n}, l_{k+2n})$ (indexes are modulo $3n$) form triangles, each of the triangles containing $x$ (by the central point property of $x$).
\end{proof}

\begin{proof}[Proof of Theorem~\ref{dualcptmes}]
First note that it is enough to prove the theorem for a connected measure. Otherwise a measure can be represented as a limit of measures, that are connected, and have the support inside some fixed compact set. Now the general case follows from the connected case by continuity and compactness reasoning, it is shown below that the set of candidates to be the central point is compact.

Take a point $x\in\mathbb R^d$ and consider the map
$$
\pi_x : \gamma_d^1\to \mathbb R^d,
$$
taking every hyperplane $H$ to the projection of $x$ onto $H$. This map is defined and bijective almost everywhere, so the measures $\lambda_x = {\pi_x}_* \mu$ are defined.

By Lemma~\ref{mesprojcont} these measures depend on $x$ continuously, they are also connected and have compact supports. If we consider the points $x$ in some compact set, then the supports of $\lambda_x$ are all contained in some compact set.

Consider now the setü $\cent\lambda_x$. It depends on $x$ continuously in the Hausdorff metric by Lemmas~\ref{cptcont1} and \ref{cptcont2}, hence there exists a continuous function $f(x)\in \cent\lambda_x$. Now we want to apply Theorem~\ref{fpbundles2} to $f(x)$, to do this we have to show that there exists $M$ such that for any $x$
$$
|x|>M\implies (x, x-f(x)) > 0.
$$

Let us apply Lemma~\ref{layermes} to the measure $\lambda_0$ putting
$\varepsilon = \dfrac{1}{2(d+1)}$, and find the corresponding $\delta$. Let the support of $\lambda_0$ be inside a ball of radius $R$. Let $M > R^2/\delta$, and let $x\in\mathbb R^d$ and $|x|> M$.

Put $H_x = \{y\in\mathbb R^d : (y,x)\ge (x,x)\}$. Let us find the measure of $\lambda_x(H_x)$. Let us transform $H_x$ by $\pi_0\circ\pi_x^{-1}$ to the set $L_x$. It is clear that (the two-dimensional case is sufficient to consider)
$$
L_x = \{y\in\mathbb R^d : (y,x)\ge 0\ \text{and}\ |y-x/2|\ge |x|/2\}.
$$
Hence the intersection $L_x\cap S$ is inside a layer of width $\delta$, and therefore $\lambda_0(L_x) = \lambda_x(H_x)< \varepsilon$. Thus $\cent\lambda_x\cap H_x = \emptyset$ and $(x,x-f(x)) > 0$. 

In this case the fixed point theorem can be applied, giving a point $x\in\cent\lambda_x$.
Similar to the proof of theorem~\ref{dualcpt} this point is exactly the needed point.
\end{proof}

\begin{proof}[Proof of Theorem~\ref{dualctr}]
The measures can be considered connected, similar to the previous proof.

Consider a linear subspace $L\subseteq\mathbb R^d$ of dimension $k+1$. The orthogonal projection onto $L$ takes affine $k$-flats in $\mathbb R^d$ to $l$-flats ($l\le k$) in $L$, and after dropping a set of measure zero, $k$-flats are mapped to $k$-flats.

This projection gives the measures $\nu_{i, L}$ on the set of hyperplanes $\gamma_L^1$, that depend continuously on $L$. As in the previous proof, take a point $x\in L$ and consider the measures $\lambda_{i, L, x}$ on $L$, obtained from $\nu_{i, L}$ by the projection $\pi_x$ in $L$. Their central points depend continuously on $x$ and $L$, and give the respective maps
$$
f_i : \gamma_d^{k+1} \to \gamma_d^{k+1}\quad (i = 1,\ldots, d-k),
$$
that satisfy the hypothesis of Theorem~\ref{fpbundles2} (taking into account Theorem~\ref{fpgrass}). By this theorem we obtain $x$ and $L$ such that $x\in\cent\lambda_{i, L, x}$ for any $i=1,\ldots, d-k$. Now it is obvious that the $d-k-1$-flat, perpendicular to $L$ and passing through $x$ is the required flat.
\end{proof}

\section{Geometric constructions in Tverberg type theorems}
\label{tvgconstr}

Remind the construction in the proof of the topological Tverberg theorem. We follow the two sources. In the book~\cite{mat2003} the construction is described for the case, when the number of subfamilies in the partition is prime. In paper~\cite{vol1996} the case of prime powers is considered with slightly different geometric constructions, it also contains the relevant facts on equivariant cohomology.

Consider a vector space $V$ and a positive integer $n$. Let us make a definition.

\begin{defn} Take $\mathbb R^n$ with coordinates $(t_1,\ldots,t_n)$ and consider its affine subspace $A_n$ defined by
$$
t_1 + t_2 + \ldots + t_n = 1.
$$
If the space $A_n$ has to be considered as linear, we put the origin to $(1/n, 1/n, \ldots, 1/n)$.
\end{defn}

\begin{defn}
For a vector space $V$ and a positive integer $n$ put
$$
J_A^n(V) = nV\oplus A_n,
$$
where $nV$ is the direct sum of $n$ copies of $V$. The space $V$ is embedded into $J_A^n(V)$ by the map
$$
v\mapsto v\oplus v\oplus\ldots\oplus v\oplus (1/n, \ldots, 1/n),
$$
thus giving an orthogonal decomposition
$$
J_A^n(V) = V\oplus D_A^n(V).
$$
\end{defn}

The space $J_A^n(V)$ also has the meaning of the affine hull of the $n$-fold join $V*V*\dots*V$.

If the group $G$ acts on the indexes $[n]$, then it acts on the $n$-fold direct sum $nV$ by permutations, and on $A_n$ by permuting the coordinates. The summand $V$ in $J_A^n(V) = V\oplus D_A^n(V)$ is fixed under this action. If the group $G$ acts transitively on $[n]$, then the representation $D_A^n(V)$ has no trivial summands.

In fact we have to consider groups $G = (Z_p)^k$, for $p$ being a prime. By identifying the set $[n]$ ($n=p^k$) with $G$ we obtain the action of $G$ on $[n]$ by shifts. The representation $D_A^n(V)$ of $G$ has no trivial summands and (see~\cite[Ch. IV, \S 1]{hsiang1975}) its Euler class $e(D_A^n(V))\neq 0\in H_G^*(\pt, Z_p)$.

\section{Topological constructions in Tverberg type theorems}
\label{tvtconstr}

Let $n=p^k$, $p$ be a prime, $G = (Z_p)^k$. Let us remind the definition of $G$-equivariant cohomology (in the sense of Borel) for a $G$-space $X$ with coefficients in a ring $A$
$$
H_G^*(X, A) = H^*\left((X\times EG/G), A\right),
$$
where $EG$ is a homotopy trivial free $G$-$CW$-complex.

Remind the structure of the ring $\Lambda_p(k) = H_G^*(\pt, Z_p))$ (see~\cite[Ch. IV, \S 1]{hsiang1975}). For $p=2$ this is a ring of $Z_2$-polynomials of $k$ generators $u_1,\ldots,u_k$ of degree $1$. For $p\not=2$ this is a tensor product of the polynomial ring in $k$ generators $u_1,\ldots,u_k$ of degree $2$, and an exterior algebra of $k$ generators $v_1,\ldots,v_k$ of degree $1$. The natural equivariant map $X\to \pt$ gives a natural map $\Lambda_p(k)\to H_G^*(X, Z_p)$ for any $G$-space $X$.

Put $S_2(k) = \Lambda_2(k)$ and $S_p(k) = Z_p[u_1,\ldots, u_k]$ for $p\not=2$, these are commutative polynomial rings. We need the following lemma.

\begin{lem}
\label{cohomologymult} For any nonzero $x\in S^n_p(k)$ and an integer $m\ge 0$ there is an element $x'\in \Lambda^m_p(k)$ such that the product $xy\neq 0\in \Lambda^{n+m}_p(k)$.
\end{lem}

The lemma follows directly from the description of the rings $\Lambda_p(k)$.

\begin{lem}
\label{equiveuler} For any $G$-representation in a linear space $V$, the space $V$ can be considered as a $G$-equivariant bundle over $\pt$. If $V$ has no trivial summands, then its Euler class modulo $p$ is in $S_p(k)$.
\end{lem}

\begin{proof}
For $p=2$ it is nothing to prove, for $p\neq 2$ any nontrivial irreducible real representation of $G$ is two-dimensional and its Euler class is a linear combination of $u_i\in S_p(k)$, see also~\cite[Ch. IV, \S 1]{hsiang1975}.
\end{proof}

In the statements below, that do not depend on the coefficients, $A$ is an arbitrary coefficient ring.

\begin{lem}
\label{cohomologydim} If $G$ acts freely on an $n$-dimensional simplicial or $CW$-complex $X$, then $H_G^m(X, A)$ is zero for $m>n$.
\end{lem}

\begin{proof}
The free action implies that $H_G^m(X, A) = H(X/G, A)$ the latter space being $n$-dimensional.
\end{proof}

We need the following lemma on the cohomology product, that is frequently used in estimating the Lyusternik-Schnirelmann category from the cup length.

\begin{lem}
\label{hom-prod} Let a topological space $X$ be covered by a family of open sets $U_1, U_2,\ldots, U_m$ and let $a_1, a_2,\ldots, a_m\in H^*(X, A)$. If for any $i=1,\ldots, m$
the image of $a_i$ in $H^*(U_i, A)$ is zero then the product $a_1a_2\cdots
a_m = 0$ in $H^*(X, A)$.
\end{lem}

The lemma is obvious, since the classes $a_i$ come from the respective $H^*(X, U_i, A)$, and the conclusion follows by the property of the cohomology multiplication.

Now let us describe the configuration spaces used to (see~\cite{mat2003}) prove the topological and colorful Tverberg theorems.

Consider the $N$-fold join of a discrete set of $n$ points $EG_N = [n]*[n]*\dots*[n]$. This space (see~\cite{mat2003}) is known to be $N-1$-dimensional and $N-2$-connected. In fact it follows from the basic properties of the join.

The action of $G$ on $[n]$ gives a free action of $G$ on $EG_N$. The natural map 
$$
H_G^*(\pt, Z_p)\to H_G^*(EG_N, Z_p)
$$
is injective in dimensions $\le N-1$, it can be shown by considering the Leray-Serre spectral sequence for $\left(EG\times EG_N\right)/G \to BG$, see~\cite[Ch.~III, \S 1]{hsiang1975}.

\begin{defn}
For a simplicial complex $K$ define the \emph{deleted} $n$-fold join
$K_\Delta^{*n}$ by taking the vertex set $V(K)\times[n]$, and the set of simplexes
$$
\sigma=\sigma_1\times\{1\}\cup\sigma_2\times\{2\} \ldots
\cup\sigma_n\times\{n\},
$$
where $\sigma_1,\ldots,\sigma_n$ are pairwise disjoint simplexes of $K$.
\end{defn}

The above space $EG_N$ is a simplicial complex, isomorphic to the \emph{deleted} join of $N-1$-dimensional simplex $(\Delta^{N-1})_\Delta^{*n}$. The latter being the natural configuration space in the topological Tverberg theorem and Theorem~\ref{dualtverbergp}. The group $G$ acts on the join by permuting the summands.

The colorful Tverberg theorem has its own configuration space. Consider the simplicial complex $K(k,t)$ with vertex set $[k]\times[t]$, the first index corresponding to the color. Let the simplexes of $K(k,t)$ be the subsets of $K(k,t)$ that contain at most one vertex of every color.

Now the configuration space for the colorful Tverberg theorem is the $r$-fold deleted join $L(k,t,r) = K(k,t)^{*r}_\Delta$. This space is $rk-2$-connected, when $t\ge 2r-1$ (see~\cite{zivvre1994}). Hence the map
$$
H_G^*(\pt, Z_p)\to H_G^*(L(k,t,r), Z_p)
$$
is injective in dimensions $\le rk-1$ from the Leray-Serre spectral sequence.

\section{Proofs of theorems~\ref{dualtverbergp} and \ref{coldualtverberg}}
\label{tvproofs}

Let us give the proof of Theorem~\ref{dualtverbergp}, the proof of Theorem~\ref{coldualtverberg} is obtained by replacing $\Delta^{n(d+1)-1}$ with $K(d+1, t)$, and the corresponding replacing of the deleted joins.

Put $N=n(d+1)$ and denote the hyperplanes $\{h_i\}_{i=1}^N$. Put, as above, $G=(Z_p)^k$.

Define the map $f_i$ to be the orthogonal projection onto $h_i$. Consider the convex body $B$ stable under all these projections.

The full simplicial complex with the vertex set corresponding to $\{h_i\}$ is identified with $\Delta^{N-1}$.

Take a point $b\in B$. Let us define the map $s_b :(\Delta^{N-1})_\Delta^{*n} = EG_N\to J_A^n(\mathbb R^d)$. First, let us define the map $t_b: \Delta^{N-1}\to \mathbb R^d$ on vertices of $\Delta^{N-1}$ as taking the $i$-th vertex to $f_i(b)-b$, and extend it to a linear map from $\Delta^{N-1}$. Let the map $s_b$ be $n$-fold join of $t_b$. Taking into account the dependence of $s_b$ on $b$, we obtain a continuous map $s :B\times (\Delta^{N-1})_\Delta^{*n} = B\times EG_N\to J_A^n(\mathbb R^d)$.

The map $s$ is $G$-equivariant, where the actions of $G$ on $EG_N$ and $J_A^n(\mathbb R^d)$ are introduced in previous sections. This map can be considered as a section of the $G$-equivariant vector bundle 
$$
B\times EG_N\times J_A^n(\mathbb R^d)\to B\times EG_N.
$$

We are going to prove that $s$ has some zeros. The equivariant Euler class $e(D_A^n(\mathbb R^d))$ has dimension $(d+1)(n-1)$ and is nonzero in $H_G^*(EG_N, Z_p)$, because the natural map of cohomology
$$
S^{(d+1)(n-1)}_p(k) \to H_G^{(d+1)(n-1)} (EG_N, Z_p)
$$
is injective. By Lemma~\ref{cohomologymult} there also exists $e'\in \Lambda^{d}(k)$ such that $e(D_A^n(\mathbb R^d))e'$ is nonzero in $H_G^{n(d+1)-1}(EG_N, Z_p)$.

Let us decompose $J_A^n(\mathbb R^d) = D_A^n(\mathbb R^d)\oplus \mathbb R^d$. The section $s$ is decomposed into respective $s_1\oplus s_2$. The relative Euler class of $s_2$ resides in $H^d(B,\partial B)$ and is nonzero, as it is in the proof of the fixed point theorem. By the multiplicativity and the K\"unneth formula the relative Euler class of $s$ is nonzero in $H_G^{N-1}(B\times EG_N, \partial B\times EG_N, Z_p)$. Hence there exist pairs $b\times y$ such that $s(b, y)=0$, denote the set of all such pairs by $Z$.

Now let us describe the meaning of inclusion $b\times y\in Z$ for $b\in B$ and $y\in EG_N$. The point $y$ is a convex combination of some points of $\Delta^{N-1}$, by the definition of the deleted join
$$
y = c_1x_1\oplus\dots c_nx_n,
$$
the points $x_i\in\Delta^{N-1}$ having pairwise disjoint supports. The equation $s(b,y)=0$ means that 
$$
c_1 = \dots = c_n = 1/n,\quad t_b(x_1)= \dots = t_b(x_n) = b.
$$ 
Let us represent $x_i$ in the barycentric coordinates in $\Delta^{N-1}$
$$
c_ix_i = a_{i1}v_1\oplus\dots\oplus a_{iN}v_N.
$$

The definition of the deleted join is equivalent to the fact, that for any $j\in[N]$ the inequality $a_{ij}\neq 0$ holds for at most one $i\in[n]$. Let us state a lemma.

\begin{lem}
\label{hyperplane-gp}
Let $J\subseteq [N]$, $F=\{t_b(v_j) : j\in J\}$. If $b\in\conv F$ then

1) either $b\in F$;

2) or $|F|\ge d+1$ and $b$ is in the interior of $\conv F$.
\end{lem}

This lemma follows from the general position hypothesis.

For a given point $b\times y\in Z$ put
$$
J_b = \{j\in[N] : f_j(b)= t_b(v_j) = b\}.
$$

Assume that for any $b\times y\in Z$ the set $J_b$ is nonempty. Define the map $p: Z \to EG_N$ as follows. Take the coordinates of $y$ to be $\{a_{ij}\}$ as above. The map $p$ forgets $b$ and turns to zero all the coordinates $a_{ij}$, except such that $j\in J_b$, and then normalizes the coordinates to have a unit sum. This map is continuous and $G$-equivariant. 

The image of $p(Z)$ is in the $d-1$-dimensional skeleton of $EG_N$, because at most $d$ of the coordinates $a_{ij}$ can be nonzero (they correspond to the hyperplanes containing $b$). This $d-1$-dimensional skeleton is $G$-$ANR$, hence $p$ can be extended to some $G$-invariant neighborhood $U\supseteq Z$.

The image of the element $e'\in \Lambda_p^d(k)$ in $H_G^d(p(U), Z_p)$ is zero by Lemma~\ref{cohomologydim}, and therefore it is zero in $H_G^d(U, Z_p)$. From the description of the Euler class as an obstruction it follows that $e(J_A^n(\mathbb R^d))$ is zero in $H_G^*(B\times EG_N\setminus Z, \partial B\times EG_N, Z_p)$. By Lemma~\ref{hom-prod} the product $e(J_A^n(\mathbb R^d))e'$ should be zero in $H_G^*(B\times EG_N, \partial B\times EG_N, Z_p)$. But it was shown above that $e(D_A^n(\mathbb R^d))e'\neq 0\in H_G(EG_N, Z_p)$, hence by the K\"unneth formula 
$$
e(J_A^n(\mathbb R^d))e'\not=0\in H_G^*(B\times EG_N, \partial B\times EG_N, Z_p).
$$

Hence the assumption was false and for some $b\times y\in Z$ the set $J_b$ is empty. Then the sets
$$
S_i = \{j\in[J] : a_{ij}\neq 0\}
$$
are disjoint, and $\conv t_b(S_i)\ni b$ for any $i\in[n]$. By Lemma~\ref{hyperplane-gp} all these sets have cardinality $d+1$ and $\int\conv t_b(S_i)\ni b$. It is easy to see that the families of hyperplanes $\{h_j\}_{j\in S_i}$ form simplexes, containing $b$ in their interiors.

\section{Some conjectures}
\label{conjectures}

Besides the conjecture, that the dual Tverberg theorem holds for any cardinality of partition, there are another conjectures, related to the above results. The first conjecture interpolates between the case $k=0$, which is the ordinary central point theorem, and $k=n-1$, which is the dual central point theorem.

\begin{con}[The central point theorem for $k$-flats]
There exists a constant $c(k, d) > 0$ with the following property. For any absolutely continuous probabilistic measure $\mu$ with compact support on the set of all affine $k$-flats in $\mathbb R^d$, there exists a point $x$ such that for any $(n-k)$-flat $M\ni x$ we have
$$
\mu(I(M, k)) \ge c(k, d).
$$
\end{con}

The following conjecture would be a dual analogue of the colorful Tverberg theorem in the plane from~\cite{bl1992}.

\begin{con}[The dual colorful Tverberg theorem in the plane]
Suppose $3n$ straight lines of general position are given in the plane. Let the lines be painted into $3$ colors, each color having $n$ lines. Then the lines can be partitioned into $n$ colorful triples so that the triangles, formed by triples, have a common point.
\end{con}

\section{Acknowledgments}

The author thanks A.Yu.~Volovikov and V.L.~Dol'nikov for the discussion of these results, and I.~B\'ar\'any for pointing out the references for Lemma~\ref{projbounded}.


\begin{thebibliography}{99}

\bibitem{adw1984}
R.~Aharoni, P.~Duchet, B.~Wajnryb. Successive projections on hyperplanes. // J. Math. Anal. Appl., 103, 1984, 134--138.

\bibitem{bl1992}
I.~B\'ar\'any, D.G.~Larman. A colored version of Tverberg's theorem. // J. London Math. Soc., 45(2), 1992, 314--320.

\bibitem{dgk1963}
L.~Danzer, B.~Gr\"unbaum, V.~Klee. Helly's theorem and its relatives. // Convexity, Proc. Symp. Pure Math. Vol. VII, Amer. Math. Soc., Providence, R.I., 1963, 101--179.

\bibitem{dol1993}
V.L.~Dol'nikov. Transversals of families of sets in $\mathbb R^n$ and a connection between the Helly and Borsuk theorems (In Russian), // Sb., Math. 79(1), 1994, 93--107; translation from Mat. Sb., 184(5), 1993, 111--132.

\bibitem{eck1993}
J.~Eckhoff. Helly, Radon, and Carath\'eodory type theorems. // Handbook of Convex Geometry, ed. by P.M.~Gruber and J.M.~Willis, North-Holland, Amsterdam, 1993, 389--448.

\bibitem{grun1960}
B.~Gr\"unbaum. Partitions of mass-distributions and of convex bodies by hyperplanes. // Pacific J. Math., 10, 1960, 1257--1261.

\bibitem{hil1980A}
H.L.~Hiller.  On the cohomology of real grassmanians. // Trans. Amer. Math. Soc., 257(2), 1980, 521--533.

\bibitem{hsiang1975}
Wu Yi Hsiang. Cohomology theory of topological transformation groups. Berlin-Heidelberg-New-York, Springer Verlag, 1975.

\bibitem{mishch1998}
G.~Luke, A.S.~Mishchenko. Vector bundles and their applications. Springer Verlag, 1998.

\bibitem{mat2003}
J.~Matou\v sek. Using the Borsuk-Ulam theorem. // Berlin-Heidelberg, Springer Verlag, 2003.

\bibitem{milsta1974}
J.~Milnor, J.~Stasheff. Characteristic classes. Princeton University Press, 1974.

\bibitem{neumann1945}
B.H.~Neumann. On an invariant of plane regions and mass distributions. // J. London Math. Soc., 20, 1945, 226--237.

\bibitem{rado1946}
R.~Rado. A theorem on general measure. // J. London Math. Soc., 21, 1946, 291--300.

\bibitem{roud1988}
J.P.~Roudneff. Tverberg-type theorems for pseudoconfigurations of points in the plane. // European Journal of Combinatorics, 9(2), 1988, 189--198.

\bibitem{ruhu1999}
P.J.~Rousseeuw, M.~Hubert. Depth in an arrangement of hyperplanes. // Discrete and Computational Geometry, 22, 1999, 167--176.

\bibitem{shg1977}
G.E.~Shilov. Integral, Measure, and Derivative: A Unified Approach. Dover Publications, 1977.

\bibitem{tver1966}
H.~Tverberg. A generalization of Radon's theorem. // J. London Math. Soc., 41, 1966, 123--128.

\bibitem{vol1996}
A.Yu.~Volovikov. On a topological generalization of the Tverberg theorem. // Mathematical Notes, 59(3), 1996, 324--326.

\bibitem{zivvre1990}
R.T.~\v Zivaljevi\'c, S.T.~Vre\'cica. An extension of the ham sandwich theorem. // Bull. London Math. Soc., 22, 1990, 183--186.

\bibitem{zivvre1994}
R.T.~\v Zivaljevi\'c, S.T.~Vre\'cica. New cases of the colored Tverberg theorem. // Jerusalem Combinatorics '93, ed. by H.~Barcelo, G.~Kalai, Contemporary Mathematics, Amer. Math. Soc., Providence, 1994, 325--334.

\end{thebibliography}
\end{document}